\documentclass[12pt, reqno]{amsart}
\usepackage{amsmath, amsthm, amscd, amsfonts, amssymb, graphicx, color}
\usepackage[bookmarksnumbered, colorlinks, plainpages]{hyperref}

\newtheorem{theorem}{Theorem}[section]
\newtheorem{lem}[theorem]{Lemma}
\newtheorem{prop}[theorem]{Proposition}
\newtheorem{cor}[theorem]{Corollary}
\theoremstyle{definition}

\theoremstyle{remark}

\numberwithin{equation}{section}

\newcommand{\HH}{\mathcal{H}}

\newcommand{\R}{\mathbb{R}}

\newcommand{\leqs}{\leqslant}
\newcommand{\geqs}{\geqslant}

\newcommand{\ap}{\alpha}

\newcommand{\ep}{\epsilon}

\newcommand{\sm}{\,\sigma\,}

\newcommand{\hm}{\, ! \,}
\newcommand{\gem}{\, \# \,}
\newcommand{\am}{\, \triangledown \,}

\begin{document}
\setcounter{page}{1}

\title[Characterizations of Operator Monotonicity via Operator Means]{Characterizations of Operator Monotonicity via Operator Means and Applications to Operator Inequalities}

\author[P. Chansangiam]{Pattrawut Chansangiam$^{*}$}

\address{$^{*}$ Department of Mathematics, Faculty of Science, King Mongkut's Institute of Technology
Ladkrabang,  Bangkok 10520, Thailand.}
\email{\textcolor[rgb]{0.00,0.00,0.84}{kcpattra@kmitl.ac.th}}



\date{Received: xxxxxx; Revised: yyyyyy; Accepted: zzzzzz.
\newline \indent $^{*}$ Corresponding author. This research
	is supported by King Mongkut's Institute
of Technology Ladkrabang Research Fund grant no. KREF045710.}

\begin{abstract}
We prove that a continuous function $f:(0,\infty) \to (0,\infty)$
is operator monotone increasing if and only if $f(A \: !_t \: B) \leqs f(A) \: !_t \: f(B)$ 
for any positive operators $A,B$ and scalar $t \in [0,1]$.
Here, $!_t$ denotes the $t$-weighted harmonic mean. 
As a counterpart, $f$ is operator monotone decreasing if and only if 
the reverse of preceding inequality holds.
Moreover, we obtain many characterizations of operator-monotone increasingness/decreasingness 
in terms of operator means.
These characterizations lead to many operator inequalities involving means.
\end{abstract}

\maketitle

\section{Introduction}

Let $B(\HH)$ be the algebra of bounded linear operators on  a complex Hilbert space $\HH$.
The cone of positive operators on $\HH$ is written by $B(\HH)^+$.
For selfadjoint operators $A,B \in B(\HH)$, the partial order $A \leqs B$ means that $B-A \in B(\HH)^+$,
while the notation $A<B$ indicates that $B-A$ is an invertible positive operator.

A useful and important class of real-valued functions is the class of operator monotone functions,
introduced by L\"{o}wner in a seminal paper \cite{Lowner}.
Let $I \subseteq \R$ be an interval. A continuous
function $f: I \to \R$ is said to be \emph{operator monotone (increasing)} if
\begin{align}
	A \leqs B \implies f(A) \leqs f(B) \label{eq: OMI}
\end{align}
for all operators $A,B \in B(\HH)$ whose spectra contained in $I$ and for all Hilbert spaces $\HH$.
If the reverse inequality in the right hand side of \eqref{eq: OMI} holds, 
then we say that $f$ is \emph{operator monotone decreasing}. 
It is well known that (see e.g. \cite[Example 2.5.9]{ Hiai}) 
the function $f(x)=x^{\ap}$ 
is operator monotone on 
$[0,\infty)$ if and only if $\ap \in [0,1]$, 
and it is operator monotone decreasing if and only if $\ap \in [-1,0]$ 
The function $x \mapsto \log (x+1)$ is operator monotone on $(0,\infty)$.
More concrete examples can be found in \cite{Furuta}.

Operator monotony arises naturally in matrix/operator inequalities 
(e.g. \cite{Ando, Bhatia_positive def matrices, Zhan}).
It plays a major role in the so-called Kubo-Ando theory of operator means (e.g. \cite{Kubo-Ando}).
It has applications in many areas, including electrical networks (see e.g. \cite{Anderson-Trapp}),
elementary particles (\cite{Wigner-von Neumann}) and entropy in physics (\cite{Fujii_entropy}).

A closely related concept to operator monotonicity is the concept of operator concavity.
A continuous function $f: (0,\infty) \to \R$ is said to be \emph{operator concave} if 
\begin{align}
	f((1-t)A+tB) \:\geqs\: (1-t)f(A) + tf(B) 
	\label{eq: operator concavity}
\end{align}
for any $A,B>0$ and $t \in [0,1]$.
The continuity of $f$ implies that $f$ is operator concave 
if and only if $f$ is operator midpoint-concave, in the sense that the condition \eqref{eq: operator concavity}
holds for $t=1/2$.

L\"{o}wner \cite{Lowner}  characterized operator monotonicity in terms of the positivity of
matrix of divided differences and an important class of analytic functions, namely,
Pick functions.
Hansen and Pedersen \cite{Hansen-Pedersen} provided 
various characterizations of operator monotonicity in terms of operator inequalities, 
using $2$-by-$2$ block matrix techniques. 
The following fact is well known:

\begin{theorem}  \label{thm: OM iff O concave} %
	(see \cite{Hansen-Pedersen} or \cite[Corollary 2.5.4]{Hiai})
	A continuous function $f:(0,\infty) \to (0,\infty)$ is operator monotone 
  if and only if
  $f$ is operator concave.
\end{theorem}

\noindent	Kubo and Ando \cite{Kubo-Ando} characterized operator monotony in terms of operator connections
(see details in the next section).
Systematic explanations of operator monotonicity/concavity can be found in  
\cite[Chapter V]{Bhatia} and \cite[Section 2]{Hiai}.

In the present paper, we focus on the relationship between the operator monotonicity of functions 
and operator means. 
See some related discussions in \cite{Ando-Hiai, Aujla et al}.
Note that the condition \eqref{eq: operator concavity} 
can be restated in terms of weighted arithmetic means $\triangledown_t$ as follows
\begin{align}
	f(A \, \triangledown_t \, B) \geqs f(A) \,\triangledown_t\, f(B)    \label{eq: OMI in terms of weighted AM}
\end{align}
for any $A,B>0$ and $t \in [0,1]$.
The reverse inequality of \eqref{eq: OMI in terms of weighted AM} 
is equivalent to the \emph{operator convexity}
of $f$.
Consider the $t$-weighted harmonic mean 
\begin{align*}
	A \,!_{t}\, B \:=\: [(1-t)A^{-1} + t B^{-1}]^{-1}, \quad A,B>0.
\end{align*}
We prove an interesting fact about operator monotone functions:
\begin{align}
	f(A \,!_t \,B) \:\leqs\: f(A) \,!_t\, f(B), \quad A,B > 0 \text{ and } t \in [0,1].
	\label{eq: f (A w hm B)}
\end{align}
Conversely, the above property characterizes the operator monotonicity of $f$.
Moreover, $f$ is operator monotone decreasing if and only if 
the reverse inequality of \eqref{eq: f (A w hm B)}
holds.
Many characterizations of operator monotone increasing/decreasing functions 
in this type are established in Section 3.
This gives a natural way to derive operator inequalities involving means in Section 4.

\section{Preliminaries on operator means}

In this section, we review Kubo-Ando theory of operator means 
(see e.g. \cite[Chapter 4]{Bhatia_positive def matrices}, \cite[Section 3]{Hiai}).
We begin with the axiomatic definition of an operator mean.
Then we mention fundamental results and give practical examples of operator means
which will be used in later discussions.

An \emph{operator connection} is a binary operation $\sm$ on $B(\HH)^+$
such that for all positive operators $A,B,C,D$:
\begin{enumerate}
	\item[(M1)] monotonicity: $A \leqs C, B \leqs D \implies A \sm B \leqs C \sm D$
	\item[(M2)] transformer inequality: $C(A \sm B)C \leqs (CAC) \sm (CBC)$
	\item[(M3)] continuity from above:  for $A_n,B_n \in B(\HH)^+$,
                if $A_n \downarrow A$ and $B_n \downarrow B$,
                 then $A_n \sm B_n \downarrow A \sm B$.
                 Here, $X_n \downarrow A$ indicates that $(X_n)$ is a decreasing sequence
                 converging strongly to $X$.
\end{enumerate}
An \emph{operator mean} is an operator connection with property that $A \sm A =A$ for all $A \geqs 0$.

Classical examples of operator means are the arithmetic mean (AM), the harmonic mean (HM),
the geometric mean (GM) and their weighted versions.
For each $t \in [0,1]$, we define the $t$-weighted arithmetic mean, the $t$-weighted harmonic mean
and the $t$-weighted geometric mean for invertible positive operators $A$ and $B$ as follows:
\begin{align*}
	A \,\triangledown_t\, B \:&=\: (1-t)A + t B \\
	A \,!_t\, B \:&=\: \left[(1-t)A^{-1} + t B^{-1}\right]^{-1} \\
	A \,\#_t\, B \:&=\: A^{\frac{1}{2}} (A^{-\frac{1}{2}} B A^{-\frac{1}{2}})^t A^{\frac{1}{2}}.
\end{align*}
These means can be extended to arbitrary $A,B \geqs 0$ by continuity. For example,
\begin{align*}
	A \,\#_t\, B \:&=\: \lim_{\ep \downarrow 0} (A+ \ep I) \,\#_t\, (B+ \ep I),
\end{align*}
here the limit is taken in the strong-operator topology.
We abbreviate $\triangledown \,=\, \triangledown_{1/2}$, $! \,=\, !_{1/2}$
and $\#\,=\,\#_{1/2}$.

A famous theorem in this theory is the one-to-one correspondence between operator connections
and operator monotone functions:

\begin{theorem} \label{thm: Kubo-Ando}
(\cite[Theorem 3.4]{Kubo-Ando}) 
	There is a one-to-one correspondence between operator connections $\sigma$
and operator monotone functions $f:[0,\infty) \to [0,\infty)$ given by the relation
\begin{align} 
	f(A) \:= I \: \sigma A  \label{f(A) = I sigma A}, \quad A \geqs 0.
\end{align}
\end{theorem}

\noindent Moreover, $\sigma$ is an operator mean if and only if $f(1)=1$.
Every connection $\sigma$ is uniquely determined on the set of invertible positive operators. 
Indeed, for  $A,B \geqs 0$ we have $A+\epsilon I, B+\epsilon I >0$ for any $\epsilon >0$ 
and then by the monotonicity (M1) and the continuity from above (M3),
\begin{align*}
	A \,\sigma\, B = \lim_{\epsilon \downarrow 0} \; (A+\epsilon I) \,\sigma\, (B + \epsilon I).
\end{align*}
Using (M2), every operator connection $\sigma$ is congruent invariant in the sense that
\begin{align}
	C(A \sm B)C \:=\: CAC \sm CBC   \label{eq: congruent invariance}
\end{align}
for any $A \geqs 0, B \geqs 0$ and $C>0$.
Theorem \ref{thm: Kubo-Ando} serves a simple proof 
of operator versions of the weighted AM-GM-HM inequalities:

\begin{prop}	\label{thm: weighted AM-GM-HM}
(see e.g. \cite[Proposition 3.3.2]{Hiai}) 
	For each $A,B \geqs 0$ and $t \in [0,1]$, we have 
	\begin{align}
		A \,\triangledown_t\, B \:\geqs\: A\,!_t\, B \:\geqs\: A\,\#_t\, B.   \label{eq: weighted AM-GM-HM}
	\end{align}
\end{prop}

An operator connection $\sigma$ is symmetric if $A \sm B = B \sm A$ for all $A,B \geqs 0$, or equivalently,
$f(x)=xf(1/x)$ for all $x >0$.
Let $\sigma$ be a nonzero operator connection.
Then $A \,\sigma\,B >0$ for any $A,B>0$ (see e.g. \cite{Chansangiam1}).
We define the \emph{adjoint} of $\sigma$ to be the operator connection
\begin{align*}
	\sigma^*: &(A,B) \mapsto (A^{-1} \,\sigma\,B^{-1})^{-1}.
\end{align*}
If $\sigma$ has the representing function $f$, then
the representing function of $\sigma^*$
is given by the following operator monotone function:
\begin{align*}
	f^* (x) &= \frac{1}{f(1/x)}, \quad x>0.
\end{align*}

Next, we introduce two important classes of parametrized means.
For each $p \in [-1,1]$ and $\ap \in [0,1]$, consider the operator monotone function
(see e.g. \cite{Bhatia})
\begin{align*}
	f_{p,\ap}(x) \:=\: (1-\ap +\ap x^p)^{1/p}, \quad x \geqs 0.
\end{align*}
We define $f_{p,\ap}(0) \equiv 0$ and $f_{p,\ap}(1) \equiv 1$ by continuity. 
When $p=0$, it is understood that we take limit at $p$ tends to $0$ and, by L'H\^{o}spital's rule,
\begin{align*}
	f_{0,\ap} (x) = x^{\ap}.
\end{align*} 
Each function $f_{p,\ap}$ gives rise to a unique operator mean, namely, the \emph{quasi-arithmetic power mean}
$\#_{p,\ap}$ with exponent $p$ and weight $\ap$. Note that
\begin{align*}
	\#_{1,\ap} \:=\: \triangledown_{\ap}, \quad \#_{0,\ap} \:=\: \#_{\ap}, \quad
	\#_{-1,\ap} \:=\: !_{\ap}.
\end{align*}

The second class is a class of symmetric means.
Let $r \in [-1,1]$ and consider the operator monotone function (see \cite{Fujii-Seo})
\begin{align*}
	g_r (x) \:=\: \frac{3r-1}{3r+1} \,\frac{x^{\frac{3r+1}{2}} -1}{x^{\frac{3r-1}{2}} -1}, \quad x\geqs 0.
\end{align*}
This function satisfies $g_r(1)=1$ and $g_r(x) = xg_r(1/x)$. 
Thus it associates to a unique symmetric operator mean, denoted by $\diamondsuit_r$.
In particular,
\begin{align*}
	\diamondsuit_1 \:=\: \triangledown, \quad
	\diamondsuit_0 \:=\: \#, \quad
	\diamondsuit_{-1} \:=\: !.
\end{align*}
The operator means $\diamondsuit_{1/3}$ and $\diamondsuit_{-1/3}$ are known as the logarithmic mean
and its adjoint.



\section{Characterizations of operator monotonicity via operator means}

In this section, we characterize operator monotone increasing/decreasing functions in terms of operator means.
Let us start with a simple observation about operator monotony:

\begin{lem} \label{lem: f anf f star}
  A function $f:(0,\infty) \to (0,\infty)$ is operator monotone 
  (operator monotone decreasing) if and only if
  $f^*$ is operator monotone (operator monotone decreasing, resp.). 
\end{lem}
\begin{proof}
	We consider only for the case of operator monotonicity 
	since a proof for another case is similar to this one.
	It is easy to see that the continuity of $f$ and $f^*$ are equivalent.
	Suppose that $f$ is operator monotone and consider $A,B>0$ such that $A \geqs B$.
	Then $A^{-1} \leqs B^{-1}$ by the operator decreasingness of the map $t \mapsto t^{-1}$ on $(0,\infty)$. 
	Hence $f(A^{-1}) \leqs f(B^{-1})$ by the operator monotonicity of $f$.
	Since $f^*(t)=f(t^{-1})^{-1}$ for all $t>0$, it follows that 
	$$ f^*(A) \:=\: f(A^{-1})^{-1} \:\geqs\: f(B^{-1})^{-1} \:=\: f^*(B).$$
	For the converse, use the fact that $(f^*)^*=f$.
\end{proof}

The following result dualizes operator inequalities and will be used many times later.

\begin{prop} \label{prop 4.1}
	Let $\sigma$ and $\eta$ be binary operations for invertible positive operators.  
	Let $f,g,h:(0,\infty) \to (0,\infty)$ be continuous functions.
	Then the following statements are equivalent:
	\begin{enumerate}
		\item[(1)]	$f(A \:\sigma\: B) \leqs g(A) \,\eta\, h(B)$ for all $A,B>0$ ;
		\item[(2)]	$f^*(A \:\sigma^*\: B) \geqs g^*(A) \:\eta^*\: h^*(B)$ for all $A,B>0$.
	\end{enumerate}
\end{prop}
\begin{proof}
	Assume (1) and consider $A,B>0$. By definition of $f^*$
	and the operator-monotone decreasingness of the map $t \mapsto t^{-1}$, we have
	\begin{align*}
		f^*(A \,\sigma^*\, B) 
		&= f^*\left((A^{-1} \sm B^{-1})^{-1}\right) \\
		&=\left[ f(A^{-1} \sm B^{-1}) \right]^{-1} \\
		&\geqs \left[ g(A^{-1}) \,\eta\, h(B^{-1}) \right]^{-1} \\
		&=\left[ g^*(A)^{-1} \,\eta\, h^*(B)^{-1} \right]^{-1} \\
		&= g^*(A) \,\eta^*\, h^*(B).
	\end{align*}	 
	To prove (2) $\Rightarrow$ (1), apply (1) $\Rightarrow$ (2) to continuous functions $f^*,g^*,h^*$
	and binary operations $\sigma^*,\eta^*$.
\end{proof}

Recall the following result.

\begin{prop} \label{thm: Ando-Hiai}
 (\cite[Theorem 2.3]{Ando-Hiai}) 
	Let $f:(0,\infty) \to (0,\infty)$ be a continuous function. The following statements are equivalent:
	\begin{enumerate}
		\item[(i)]	$f$ is operator monotone ;
		\item[(ii)]	$f(A \am B) \geqs f(A) \gem f(B)$ for all $A,B>0$ ;
		\item[(iii)]	$f(A \am B) \geqs f(A) \,\sigma\, f(B)$ for all $A,B>0$ 
					and for all symmetric means $\sigma$ ;
		\item[(iv)]	$f(A \am B) \geqs f(A) \,\sigma\, f(B)$ for all $A,B>0$ 
					and for some symmetric mean $\sigma \neq \, !$.
	\end{enumerate}
\end{prop}

The next theorem further characterizes operator monotonicity.

\begin{theorem} \label{thm: OMI}
	Let $f:(0,\infty) \to (0,\infty)$ be a continuous function.
	Then the following statements are equivalent:
	\begin{enumerate}
		\item[(I1)]	$f$ is operator monotone (increasing) ;
		\item[(I2)]	$f(A \:\triangledown\: B) \geqs f(A) \:\triangledown\: f(B)$ 
		for all $A,B>0$ ;
		\item[(I3)]	$f(A \:\triangledown_t\: B) \geqs f(A) \:\triangledown_t\: f(B)$ 
		for all $A,B>0$ 
				and for all $t \in [0,1]$ ;
		\item[(I4)]	$f(A \:!\: B) \leqs f(A) \:\#\: f(B)$  for all $A,B>0$ ;
		\item[(I5)]	$f(A \:!_t\: B) \leqs f(A) \:\#_t\: f(B)$  for all $A,B>0$ 
				and for all $t \in [0,1]$ ;
		\item[(I6)]	$f(A \:!\: B) \leqs f(A) \:!\: f(B)$  for all $A,B>0$ ;
		\item[(I7)]	$f(A \:!_t\: B) \leqs f(A) \:!_t\: f(B)$  for all $A,B>0$ 
				and for all $t \in [0,1]$ ;
		\item[(I8)]	$f(A \:!\: B) \leqs f(A) \:\sigma\: f(B)$ for all $A,B>0$ 
				and for all symmetric means $\sigma$ ;
		\item[(I9)]	$f(A \:!\: B) \leqs f(A) \:\sigma\: f(B)$ for all $A,B>0$ 
				and for some symmetric mean $\sigma \neq \triangledown$.
	\end{enumerate}
\end{theorem}
\begin{proof}
	The following implications are clear: (I3) $\Rightarrow$ (I2), (I5) $\Rightarrow$ (I4) 
	and (I7) $\Rightarrow$ (I6).
	The equivalence (I1) $\Leftrightarrow$ (I3) is a restatement of Theorem \ref{thm: OM iff O concave}.
	
	(I2) $\Rightarrow$ (I1). 
	Assume (I2), i.e. $f$ is operator midpoint-concave.
	The continuity of $f$ implies that $f$ is operator concave.
	Hence $f$ is operator monotone by Theorem \ref{thm: OM iff O concave}.
	
	(I1) $\Leftrightarrow$ (I4). 
	The operator monotonicity of $f$ and $f^*$ are equivalent by Lemma \ref{lem: f anf f star}.
	By the equivalent (i) $\Leftrightarrow$ (ii) in Proposition \ref{thm: Ando-Hiai}, 
	the operator monotonicity of $f^*$ reads
	\begin{align*}
		f^*(A \am B) \:\geqs\: f^*(A) \gem f^*(B) \quad \text{ for all } A,B>0.
	\end{align*}
	Proposition \ref{prop 4.1} says that this condition is equivalent to
	 $f(A \hm B) \leqs f(A) \gem f(B)$ for all $A,B>0$ since $\triangledown^*\,=\,!$ and $\#^*=\#$.
	
	(I1) $\Rightarrow$ (I5).
	Assume that $f$ is operator monotone. Then so is $f^*$ by Lemma \ref{lem: f anf f star}.
	The implication (I1) $\Rightarrow$ (I3) assures that for any $A,B>0$ and $t \in [0,1]$
	\begin{align*}
		f^*(A\, \triangledown_t \, B) \geqs f^*(A) \,\triangledown_t\, f^*(B).
	\end{align*}	 
	Hence, by Proposition \ref{prop 4.1}, we obtain (I5) since $\triangledown_t^*\,=\,!_t$ and $\#_t^*=\#_t$.
	
	(I1) $\Leftrightarrow$ (I6). Note that the operator monotonicity of $f$ and $f^*$ 
	are equivalent by Lemma \ref{lem: f anf f star}.
	Using the equivalence (I1) $\Leftrightarrow$ (I2), the operator monotonicity of $f^*$ can be expressed
	as
	\begin{align*}
		f^*(A \am B) \geqs f^*(A) \am f^*(B) \quad \text{ for all } A,B>0.
	\end{align*}
	Proposition \ref{prop 4.1} asserts that this condition is equivalent to
	 $f(A \hm B) \leqs f(A) \hm f(B)$ for all $A,B>0$.
	 
	(I1) $\Leftrightarrow$ (I7). The proof is similar to that of (I1) $\Leftrightarrow$ (I6).
	In this case, we utilize the equivalence (I1) $\Leftrightarrow$ (I3).
	 
	 (I1) $\Leftrightarrow$ (I8). By using (i) $\Leftrightarrow$ (iii) in Proposition \ref{thm: Ando-Hiai},   
	 the operator monotonicity of $f$ (hence, of $f^*$)
	 is equivalent to the condition that
	 \begin{align*}
		f^*(A \am B) \geqs f^*(A) \,\sigma\, f^*(B) 
	\end{align*}
	for $A,B>0$ and for all symmetric means $\sigma$.
	By Proposition \ref{prop 4.1}, this condition is then equivalent to the following:
	\begin{align*}
		f(A \hm B) \leqs f(A) \,\sigma^*\,f(B)
	\end{align*}
	for all symmetric means $\sigma$. 
	Since the map $\sigma \mapsto \sigma^*$ is bijective on the set of symmetric means, we arrive at (I8).
	
	(I1) $\Leftrightarrow$ (I9). The proof is similar to that of (I1) $\Leftrightarrow$ (I8).
	Here, we use (i) $\Leftrightarrow$ (iv) in Proposition \ref{thm: Ando-Hiai} and the fact that 
	$!^*=\triangledown$.
\end{proof}

Next, we turn to operator monotone decreasingness.
Recall the following result: 

\begin{prop} \label{thm: Ando-Hiai_decreasing} 
(\cite[Theorems 2.1 and 3.1]{Ando-Hiai})
	Let $f:(0,\infty) \to (0,\infty)$ be a continuous function. The following statements are equivalent:
	\begin{enumerate}
		\item[(i)]	$f$ is operator monotone decreasing ;
		\item[(ii)]	$f(A \am B) \leqs f(A) \gem f(B)$ for all $A,B>0$ ;
		\item[(iii)]  $f(A \am B) \leqs f(A) \,\sigma\, f(B)$ for all $A,B>0$ 
					and for all symmetric means $\sigma$ ;
		\item[(iv)]	 $f(A \am B) \leqs f(A) \,\sigma\, f(B)$ for all $A,B>0$ 
					and for some symmetric mean $\sigma \neq \,\triangledown$.
		\item[(v)] $f$ is operator convex and $f$ is decreasing.
	\end{enumerate}
\end{prop}

The next theorem is a counterpart of Theorem \ref{thm: OMI}.

\begin{theorem} \label{thm: OMD}
	Let $f:(0,\infty) \to (0,\infty)$ be a continuous function.
	Then the following statements are equivalent:
	\begin{enumerate}
		\item[(D1)]	$f$ is operator monotone decreasing ;
		\item[(D2)]	$f(A \:!\: B) \geqs f(A) \:\#\: f(B)$  for all $A,B>0$ ;
		\item[(D3)]	$f(A \:!_t\: B) \geqs f(A) \:\#_t\: f(B)$  for all $A,B>0$ 
				and for all $t \in [0,1]$ ;
		\item[(D4)]	$f(A \:!\: B) \geqs f(A) \:!\: f(B)$  for all $A,B>0$ 
		         and $f$ is decreasing ; ;
		\item[(D5)]	$f(A \:!_t\: B) \geqs f(A) \:!_t\: f(B)$  for all $A,B>0$ 
				and for all $t \in [0,1]$ and $f$ is decreasing ;
		\item[(D6)]	$f(A \:!\: B) \geqs f(A) \:\sigma\: f(B)$ for all $A,B>0$ 
				and for all symmetric means $\sigma$ ;
		\item[(D7)]	$f(A \:!\: B) \geqs f(A) \:\sigma\: f(B)$ for all $A,B>0$ 
				and for some symmetric mean $\sigma \neq\, !$.
	\end{enumerate}
\end{theorem}

\begin{proof}
	It is clear that (D3) $\Rightarrow$ (D2), (D5) $\Rightarrow$ (D4) and (D6) $\Rightarrow$ (D7).
	 
	(D1) $\Leftrightarrow$ (D2). 
	The operator-monotone decreasingness of $f$ and $f^*$ are equivalent.
	According to the equivalent (i) $\Leftrightarrow$ (ii) in Proposition \ref{thm: Ando-Hiai_decreasing}, 
	the fact that $f^*$ is operator monotone decreasing is equivalent to 
	\begin{align*}
		f^*(A \am B) \leqs f^*(A) \gem f^*(B)
	\end{align*}
	for all $A,B>0$. Proposition \ref{prop 4.1} states that this condition is equivalent to
	 $f(A \hm B) \geqs f(A) \gem f(B)$ for all $A,B>0$.

	(D1) $\Rightarrow$ (D3). Assume that $f$ is operator monotone decreasing. 
	Then so is $f^*$ by Lemma \ref{lem: f anf f star}.
	Consider $A,B>0$ and $t \in [0,1]$. 
	Recall the weighted AM-GM inequality for operators (Proposition \ref{thm: weighted AM-GM-HM}):
	\begin{align*}
		A \,\triangledown_t\, B \:\geqs\: A \,\#_t\, B.
	\end{align*}
	It follows that $f^*(A \,\triangledown_t\, B) \leqs f^*(A \,\#_t\, B)$.
	By Proposition \ref{prop 4.1}, we have $f(A \,!_t\, B) \geqs f(A \,\#_t\, B)$.
	 
	 (D1) $\Rightarrow$ (D4). Assume that $f$ is operator monotone decreasing.
	By (D2) and the GM-HM inequality for operators (Proposition \ref{thm: weighted AM-GM-HM}), we have 
	\begin{align*}
		f(A \hm B) \geqs f(A) \gem f(B) \geqs f(A) \hm f(B)
	\end{align*}
	for all $A,B>0$. It is trivial that $f$ is decreasing in usual sense.
	
	(D4) $\Rightarrow$ (D1). Suppose that $f(A \:!\: B) \geqs f(A) \:!\: f(B)$  for all $A,B>0$ 
	and $f$ is decreasing. Then $f^*$ is also a decreasing function.
	By Proposition \ref{prop 4.1}, we have $f^*(A \am B) \leqs f^*(A) \am f^*(B)$ for all $A,B>0$,
	i.e. $f^*$ is operator convex.
	The implication (v) $\Rightarrow$ (i) in Proposition \ref{thm: Ando-Hiai_decreasing} tells us that
	$f^*$ is operator monotone decreasing. Hence, so is $f$ by Lemma \ref{lem: f anf f star}.
	
	(D1) $\Rightarrow$ (D5). The proof is similar to that of (D1) $\Rightarrow$ (D4).
	Here, we use the weighted GM-HM inequality \eqref{eq: weighted AM-GM-HM}: 
	\begin{align*}
		A \,\#_t\, B \:\geqs\: A \,!_t\, B 
	\end{align*}
	for any $A,B>0$ and $t \in [0,1]$.
	
	(D1) $\Rightarrow$ (D6). Assume that $f$ is operator monotone decreasing.
	Then so is $f^*$ by Lemma \ref{lem: f anf f star}. 
	By applying the equivalence (i) $\Leftrightarrow$ (iii) 
	in Proposition \ref{thm: Ando-Hiai_decreasing} to $f^*$, 
	we obtain that $f^*(A \am B) \leqs f^*(A) \,\sigma\, f^*(B)$ for all $A,B>0$ 
	and for all symmetric means $\sigma$.
	Now, use Proposition \ref{prop 4.1}.
	
	(D7) $\Rightarrow$ (D1). This is a combination of Proposition \ref{prop 4.1}, 
	 the equivalence (i) $\Leftrightarrow$ (iv) in Proposition \ref{thm: Ando-Hiai_decreasing}
	 and Lemma \ref{lem: f anf f star}.
\end{proof}

\section{Applications to Operator Inequalities Involving Means}

Let us derive operator inequalities involving means by making use of
Theorems \ref{thm: OMI} and \ref{thm: OMD}.
A simple way is to take a specific operator monotone increasing/decreasing function.

\begin{cor}
	For each $A,B \geqs 0$
	and $\ap,t \in [0,1]$, we have
	\begin{align*}
		(A \: !_t \: B)^{\ap} \:\leqs\: A^{\ap} \: !_t \: B^{\ap}.
	\end{align*}
\end{cor}
\begin{proof}
	By applying Theorem \ref{thm: OMI} (I1) $\Rightarrow$ (I7) 
	to the operator monotone function $f(x)=x^{\ap}$, we get
	\begin{align*}
		(A \: !_t \: B)^{\ap} \:\leqs\: A^{\ap} \: !_t \: B^{\ap}
	\end{align*}
	for any $A,B>0$. 
	For general $A,B \geqs 0$, consider $A+ \ep I, B+\ep I>0$ for $\ep>0$ 
	and then use the continuity from above (M3).
\end{proof}

\begin{cor}
For each $A, B >0$
	and $t \in [0,1]$, we have
	\begin{align*}
		\log \left((A \: !_t \: B) +I \right) \:\leqs\: \log (A+I) \: !_t \: \log (B+I).
	\end{align*}
\end{cor}
\begin{proof}
	Apply Theorem \ref{thm: OMI} (I1) $\Rightarrow$ (I7) to the operator monotone function $f(x)=\log (x+1)$.
\end{proof}

\begin{theorem} \label{thm: A sm (B weighted hm C)}
	Let $\sigma$ be an operator connection. Then for any $A,B,C \geqs 0$ and $t \in [0,1]$, we have
	\begin{align}
		A \sm (B \,!_t\, C) \:&\leqs\: (A \sm B) \,!_t\, (A \sm C), \label{eq: 1} \\
		A \sm (B \,\triangledown_t\, C) \:&\geqs\: (A \sm B) \,\triangledown_t\, (A \sm C). \label{eq: 2}
	\end{align}
\end{theorem}
\begin{proof}
	It is trivial when $\sigma$ is the zero connection.
	Suppose that $\sigma$ is nonzero. 
	By Theorem \ref{thm: Kubo-Ando}, there is an operator monotone function 
	$f:[0,\infty) \to [0,\infty)$ such that $f(A) = I \sm A$ for any positive operator $A$.
	Note that $f(x)>0$ for any $x>0$; otherwise $\sigma$ is the zero connection.
	This implies that it suffices to consider $f:(0,\infty) \to (0,\infty)$.
	Theorem \ref{thm: OMI} (I1) $\Rightarrow$ (I7) implies that for each $A,B>0$, 
	\begin{align*}
		I \sm (A \,!_t \, B) = f(A \,!_t \, B) \leqs f(A) \,!_t\, f(B) = (I \sm A) \,!_t\, (I \sm B). 
	\end{align*}
	It follows from the property \eqref{eq: OMI} and the congruent invariance 
	\eqref{eq: congruent invariance}
	that,  for $A,B,C>0$, 
	\begin{align*}
		A \sm (B \,!_t\, C) 
		\:&=\: A^{\frac{1}{2}} I A^{\frac{1}{2}} \sm 
			(A^{\frac{1}{2}} A^{-\frac{1}{2}} B A^{-\frac{1}{2}} A^{\frac{1}{2}}
			\:!_t\: A^{\frac{1}{2}} A^{-\frac{1}{2}} C A^{-\frac{1}{2}} A^{\frac{1}{2}}) \\
		\:&=\: A^{\frac{1}{2}}\left[ I \sm (A^{-\frac{1}{2}} B A^{-\frac{1}{2}} \:!_t\: 
			A^{-\frac{1}{2}} C A^{-\frac{1}{2}}) \right] A^{\frac{1}{2}} \\
		\:&\leqs\:  A^{\frac{1}{2}}\left[(I \sm A^{-\frac{1}{2}} B A^{-\frac{1}{2}}) \,!_t\,
			(I \sm A^{-\frac{1}{2}} C A^{-\frac{1}{2}}) \right] A^{\frac{1}{2}} \\
		\:&=\: (A \sm B) \,!_t\, (A \sm C).
	\end{align*}
	Finally, for general $A,B,C \geqs 0$, use the continuity from above.
	The proof of the inequality \eqref{eq: 2} is similar to \eqref{eq: 1}.
	In this case, we use Theorem \ref{thm: OMI} (I1) $\Rightarrow$ (I3).
\end{proof}

	Many interesting operator inequalities are obtained as special cases 
	of Theorem \ref{thm: A sm (B weighted hm C)}. 
	For example, for any $p \in [-1,1]$ and $\ap,t \in [0,1]$, we have
	\begin{align*}
		A \,\#_{p, \ap}\, (B \,!_t\, C) \:&\leqs\: (A \,\#_{p, \ap}\, B) \,!_t\, (A \,\#_{p, \ap}\, C), \\
		A \,\#_{p, \ap}\, (B \,\triangledown_t\, C) \:&\geqs\: (A \,\#_{p, \ap}\, B) 
		\,\triangledown_t\, (A \,\#_{p, \ap}\, C)
	\end{align*}
	hold for all positive operators $A,B,C$.
	
The next theorem is a symmetric counterpart of Theorem \ref{thm: A sm (B weighted hm C)}. 

\begin{theorem}
	Let $\sigma$ be an operator connection and $\eta$ a symmetric operator mean. 
	For each $A,B,C \geqs 0$, we have
	\begin{align*}
		A \,\sm\, (B \,!\, C) \:&\leqs\: (A \sm B) \,\eta\, (A \sm C), \\
		A \,\sm\, (B \,\triangledown\, C) \:&\geqs\: (A \sm B) \,\eta\, (A \sm C).
	\end{align*}
\end{theorem}
\begin{proof}
	The proof is similar to that of Theorem \ref{thm: A sm (B weighted hm C)}.
	In this case, we apply Theorem \ref{thm: OMI} (I1) $\Rightarrow$ (I8).
\end{proof}

Consider the symmetric mean $\diamondsuit_r$ for each $r \in [-1,1]$.
	The previous theorem implies that for any operator connection $\sigma$ and positive operators $A,B$
	\begin{align*}
		A \,\sm\, (B \,!\, C) \:&\leqs\: (A \sm B) \,\diamondsuit_r\, (A \sm C), \\
		A \,\sm\, (B \,\triangledown\, C) \:&\geqs\: (A \sm B) \,\diamondsuit_r\, (A \sm C).
	\end{align*}

\begin{cor} \label{cor: before last}
	For each $A,B \geqs 0$, $r \in [-1,0]$
	and $t \in [0,1]$, we have
	\begin{align*}
		(A \: !_t \: B)^r &\:\geqs\: A^r \: \#_t \: B^r \:\geqs\: A^r \: !_t \: B^r. 
	\end{align*}
\end{cor}
\begin{proof}
	The first inequality is obtained when applying Theorem \ref{thm: OMD} (D1) $\Rightarrow$ (D3)
	to the operator monotone decreasing function $f(x)=x^r$
	The second inequality comes from the weighted GM-HM inequality \eqref{eq: weighted AM-GM-HM}.
\end{proof}

Our final result is a symmetric counterpart of Corollary \ref{cor: before last}.

\begin{cor}
	Let $\sigma$ be a symmetric operator mean. For each $A,B \geqs 0$ and $r \in [-1,0]$, we have
	\begin{align*}
		(A \: ! \: B)^r &\:\geqs\: A^r \: \sigma \: B^r. 
	\end{align*}
\end{cor}
\begin{proof}
	Apply Theorem \ref{thm: OMD} (D1) $\Rightarrow$ (D6)
	to the operator monotone decreasing function $f(x)=x^r$.
\end{proof}

\end{document}